\documentclass[12pt,a4paper,reqno,twoside]{amsart}

\usepackage[english]{babel}
\usepackage{stmaryrd}
\usepackage{dsfont}
\usepackage[symbol*,ragged]{footmisc}
\usepackage[colorlinks,linkcolor=red,anchorcolor=blue,citecolor=blue,urlcolor=blue]{hyperref}
\usepackage{color,xcolor}

\usepackage{geometry}
\usepackage{amssymb}
\usepackage{amsmath}
\usepackage{mathrsfs}
\usepackage{amsfonts}
\usepackage{epsfig}

\usepackage{amsthm}
\usepackage{amsxtra}
\usepackage{bbding}
\usepackage{epsfig}
\usepackage{graphicx}
\usepackage{latexsym}
\usepackage{mathbbol}
\usepackage{bbold}

\usepackage{pifont}
\usepackage{wasysym}
\usepackage{skull}
\usepackage{float}

\DeclareSymbolFontAlphabet{\mathbb}{AMSb}
\DeclareSymbolFontAlphabet{\mathbbol}{bbold}

\usepackage{amscd}
\usepackage[all]{xy}
\allowdisplaybreaks[4]
\usepackage{setspace}

\theoremstyle{plain}
\newtheorem{theorem}{\normalfont\scshape Theorem}[section]
\newtheorem{proposition}{\normalfont\scshape Proposition}[section]
\newtheorem{lemma}[proposition]{\normalfont\scshape Lemma}

\newtheorem*{corollary*}{\normalfont\scshape Corollary}

\newtheorem*{remark*}{\normalfont\scshape Remark}

\theoremstyle{remark}
\newtheorem*{notation}{\normalfont\scshape Notation}

\numberwithin{equation}{section}
\addtocounter{footnote}{1}

\renewcommand{\footnoterule}{
  \kern -3pt
  \hrule width 2.5in height 0.4pt
  \kern 3pt
}

\makeatletter
\@ifundefined{MakeUppercase}{}{}
\makeatother

\begin{document}
	
\title[Distribution of $\alpha p^2$ modulo one in the intersection of two P--S sets]
	  {On the distribution of $\alpha p^2$ modulo one in the intersection of two Piatetski--Shapiro sets}

\author[Junyi Chu, Jinjiang Li, Min Zhang]
       {Junyi Chu \quad \& \quad Jinjiang Li \quad \& \quad Min Zhang}

\address{[Junyi Chu] School of Applied Science, Beijing Information Science and Technology University,
		 Beijing 100192, People's Republic of China}

\email{junyi.chu.math@gmail.com}

\address{[Jinjiang Li] (Corresponding author) Department of Mathematics, China University of Mining and Technology,
         Beijing 100083, People's Republic of China}

\email{jinjiang.li.math@gmail.com}

\address{[Min Zhang] School of Applied Science, Beijing Information Science and Technology University,
         Beijing 100192, People's Republic of China}

\email{min.zhang.math@gmail.com}

\date{}

\footnotetext[1]{Jinjiang Li is the corresponding author. 

\quad\,\,
{\textbf{MR(2020) Subject Classification}}: 11J71, 11N05, 11N80, 11L07, 11L20

}

\begin{abstract}
Let $\lfloor t\rfloor$ denote the integer part of $t\in\mathbb{R}$ and $\|x\|$ the distance from $x$ to the nearest integer. Suppose that $1/2<\gamma_2<\gamma_1<1$ are two fixed constants. In this paper, it is proved  that, whenever $\alpha$ is an irrational number and $\beta$ is any real number, there exist infinitely many prime numbers $p$ in the intersection of two Piatetski--Shapiro sets, i.e., $p=\lfloor n_1^{1/\gamma_1}\rfloor=\lfloor n_2^{1/\gamma_2}\rfloor$, such that
\begin{equation*}
\|\alpha p^2+\beta\|<p^{-\frac{14(\gamma_1+\gamma_2)-27}{43}+\varepsilon},
\end{equation*}
provided that $27/14<\gamma_1+\gamma_2<2$. This result constitutes an generalization upon the previous result of Dimitrov \cite{Dimitrov-2025}.
\end{abstract}

\maketitle

\section{Introduction and Main Result}
Let $\alpha$ be irrational number, $\beta$ be real and let $\|x\|$ denote the distance from $x$ to the nearest integer. In 1947, Vinogradov \cite{Vinogradov-1947} proved that, for $\theta=1/5-\varepsilon$, there exist infinitely many primes $p$ such that
\begin{equation}\label{upper-bound-ap+b}
\|\alpha p+\beta\|<p^{-\theta}.
\end{equation}
Subsequently, in 1977, Vaughan \cite{Vaughan-1977} obtained $\theta=1/4$ with an additional factor $(\log p)^8$ on the right--hand side of (\ref{upper-bound-ap+b}). In 1983, Harman \cite{Harman-1983} introduced sieve method into this problem and proved that $\theta=3/10$. Jia \cite{Jia-1993} later improved the exponent to $\theta=4/13$. In 1996, Harman \cite{Harman-1996} made further innovations with the sieve method and obtained $\theta=7/22$. In 2000, Jia \cite{Jia-2000} developed the techniques in Harman \cite{Harman-1996} to get $\theta=9/28$. In 2002, Heath--Brown and Jia \cite{Heath-Brown-Jia-2002} evaluated the asymptotic properties of bilinear sums of this problem and established that $\theta=16/49-\varepsilon$ is admissible. The hitherto best result in this direction is due to Matom\"{a}ki \cite{Matomaki-2009} with $\theta=1/3-\varepsilon$ and $\beta=0$. For the nonlinear case,
many mathematicians investigate the distribution of $\alpha p^k$ modulo one, i.e.,
\begin{equation}\label{upper-bound-ap^k+b}
\|\alpha p^k+\beta\|<p^{-\theta_k+\varepsilon}.
\end{equation}
In 1981, Ghosh \cite{Ghosh-1981} firstly consider the quadratic case, and showed that $\theta_2=1/8$ is admissible.
Subsequently, in 1991, Baker and Harman \cite{Baker-Harman-1991} sharpened and generalized the result of Ghosh \cite{Ghosh-1981}, who demonstrated that
\begin{equation}
 \theta_k=
 \begin{cases}
 \quad 3/20, & \textrm{if $k=2$},\\
 (3\times2^{k-1})^{-1}, & \textrm{if $k\geqslant3$}.
 \end{cases}
\end{equation}

Let $\gamma\in(\frac{1}{2},1)$ be a fixed real number. The Piatetski--Shapiro sequences are sequences of the form
\begin{equation*}
 \mathscr{N}_{\gamma}:=\big\{\lfloor n^{1/\gamma}\rfloor:\,n\in \mathbb{N}^+\big\}.
\end{equation*}
Such sequences have been named in honor of Piatetski--Shapiro, who \cite{Piatetski-Shapiro-1953}, in
1953, proved that $\mathscr{N}_{\gamma}$ contains infinitely many primes provided that $\gamma\in(\frac{11}{12},1)$. The prime numbers of the form $p=\lfloor n^{1/\gamma}\rfloor$ are called \textit{Piatetski--Shapiro primes of type $\gamma$}. More precisely, for such $\gamma$ Piatetski--Shapiro \cite{Piatetski-Shapiro-1953} showed that the counting function
\begin{equation*}
 \pi_\gamma(x):=\#\big\{\textrm{prime}\,\, p\leqslant x:\,p=\lfloor n^{1/\gamma}\rfloor\,\,\textrm{for some}\,\,
 n\in\mathbb{N}^+ \big\}
\end{equation*}
satisfies the asymptotic property
\begin{equation*}
\pi_{\gamma}(x)=\frac{x^{\gamma}}{\log x}(1+o(1))
\end{equation*}
as $x\to\infty$. Since then, the range for $\gamma$ of the above asymptotic formula in which it is known that $\mathscr{N}_{\gamma}$ contains infinitely many primes has been enlarged many times (see the literatures \cite{Kolesnik-1967,Leitmann-1975,Leitmann-1980,Heath-Brown-1983,Kolesnik-1985,Liu-Rivat-1992,
Rivat-1992,Rivat-Sargos-2001}) over the years and is currently known to hold for all $\gamma\in(\frac{2426}{2817},1)$ thanks to Rivat and Sargos \cite{Rivat-Sargos-2001}. Rivat and Wu \cite{Rivat-Wu-2001} also showed that there exist infinitely many Piatetski--Shapiro primes for $\gamma\in(\frac{205}{243},1)$ by showing a lower bound of $\pi_\gamma(x)$ with the expected order of magnitude. We remark that if $\gamma>1$ then $\mathscr{N}_\gamma$ contains all natural numbers, and hence all primes, particularly.

In 1982, Leitmann \cite{Leitmann-1982}, who investigated the prime number theorem in the intersection of two Piatetski--Shapiro sets, proved that $\mathscr{N}_{\gamma_1}\cap\mathscr{N}_{\gamma_2}$ contains infinitely many primes provided that $55/28<\gamma_1+\gamma_2<2$, where $1/2<\gamma_2<\gamma_1<1$ are two fixed constants. More precisely, Leitmann \cite{Leitmann-1982} showed that, for $1/2<\gamma_2<\gamma_1<1$ with $55/28<\gamma_1+\gamma_2<2$, the counting function
\begin{equation*}
\pi(x;\gamma_1,\gamma_2):=\#\big\{\textrm{prime}\,\, p\leqslant x:\,p=\lfloor n_1^{1/\gamma_1}\rfloor=
\lfloor n_2^{1/\gamma_2}\rfloor\,\,\textrm{for some}\,\,
 n_1,n_2\in\mathbb{N}^+ \big\}
\end{equation*}
satisfies the asymptotic property
\begin{equation}\label{2-P-S-PNT}
\pi(x;\gamma_1,\gamma_2)
=\frac{\gamma_1\gamma_2}{\gamma_1+\gamma_2-1}\cdot\frac{x^{\gamma_1+\gamma_2-1}}{\log x}(1+o(1))
\end{equation}
as $x\to\infty$. In 1983, Sirota \cite{Sirota-1983} proved that (\ref{2-P-S-PNT}) holds for $31/16<\gamma_1+\gamma_2<2$. In 2014, Baker \cite{Baker-2014} established that (\ref{2-P-S-PNT}) holds for  $23/12<\gamma_1+\gamma_2<2$. The hitherto best result in this direction is due to Li, Zhai and Li \cite{Li-Zhai-Li-2026}, who showed that (\ref{2-P-S-PNT}) holds for $21/11<\gamma_1+\gamma_2<2$.

In 2023, Dimitrov \cite{Dimitrov-2023} considered the hybrid problem of the distribution of $\alpha p$ modulo one
with $p$ constrained to Piatetski--Shapiro primes of type $\gamma$. To be specific, Dimitrov \cite{Dimitrov-2023}
proved that, for fixed $11/12<\gamma<1$, there exist infinitely many Piatetski--Shapiro primes of type $\gamma$ such that
\begin{equation*}
\|\alpha p+\beta\|< p^{-\frac{12\gamma-11}{26}}(\log p)^6.
\end{equation*}
Recently, Dimitrov \cite{Dimitrov-2025} strengthened the result of Ghosh \cite{Ghosh-1981}, who proved that, for any fixed $13/14<\gamma<1$, there exist infinitely many Piatetski--Shapiro primes $p$ of type $\gamma$ such
that
\begin{equation*}
\|\alpha p^2+\beta\|< p^{-\frac{14\gamma-13}{43}+\varepsilon}.
\end{equation*}
Based on the previous results, in 2024, Li, Li and Zhang \cite{Li-Li-Zhang-2024} generalized the result of Dimitrov
\cite{Dimitrov-2023} by solving (\ref{upper-bound-ap+b}) with primes $p=\lfloor n_1^{1/\gamma_1}\rfloor=\lfloor n_2^{1/\gamma_2}\rfloor$, where $1/2<\gamma_2<\gamma_1<1$ subject to $23/12<\gamma_1+\gamma_2<2$, and with
$\theta=-\frac{12(\gamma_1+\gamma_2)-23}{38}+\varepsilon$.

In this paper, motivated by the result of Dimitrov \cite{Dimitrov-2023,Dimitrov-2025}, we shall concentrate on investigating the hybrid problem of the distribution of $\alpha p^2$ modulo one with $p$ constrained to the intersection of two Piatetski--Shapiro prime sets, and establish the following theorem.

\begin{theorem}\label{Theorem}
Let $\alpha$ be an irrational number and $\beta$ an arbitrary real number. Suppose that $1/2<\gamma_2<\gamma_1<1$ with $27/14<\gamma_1+\gamma_2<2$. Then there exist infinitely many primes $p$, which are in the intersection of two Piatetski--Shapiro sets, i.e., $p=\lfloor n^{1/\gamma_1}_1\rfloor=\lfloor n^{1/\gamma_2}_2\rfloor$, such that
\begin{equation*}
\|\alpha p^2+\beta\|<p^{-\frac{14(\gamma_1+\gamma_2)-27}{43}+\varepsilon}.
\end{equation*}
\end{theorem}

\begin{remark*}
In Theorem \ref{Theorem}, if we take $\gamma_1=1$, then the range of $\gamma_2$ reduce to $13/14<\gamma_2<1$, and
the conclusion becomes that there exist infinitely many Piatetski--Shapiro primes of type $\gamma_2$, i.e., $p=\lfloor n^{1/\gamma_2}\rfloor$, such that
\begin{equation*}
\|\alpha p^2+\beta\|< p^{-\frac{14\gamma_2-13}{43}+\varepsilon},
\end{equation*}
which keeps the same strength as the result of Dimitrov \cite{Dimitrov-2025}.
\end{remark*}

\begin{notation}
Throughout this paper, let $p$, with or without subscripts, always denote a prime number. We use $\lfloor t\rfloor,\{t\}$ and $\|t\|$ to denote the integral part of $t$, the fractional part of $t$ and the distance from $t$ to the nearest integer, respectively. As usual, denote by $\Lambda(n)$ and $\tau_k(n)$ the von Mangoldt's function and the $k$--dimensional divisor function, respectively. We write $\psi(t)=t-\lfloor t\rfloor-1/2,e(x)=e^{2\pi ix}$. The notation $n\sim N$ means $N<n\leqslant 2N$ and  $f(x)\ll g(x)$ means  $f(x)=O(g(x))$.  Denote by $c_j$ the positive constants which depend at most on $\gamma_1$ and $\gamma_2$. We use $\mathbb{N}^+,\mathbb{Z}$ and $\mathbb{R}$ to denote the set of positive natural number, the set of integer, and the set of real number, respectively.
\end{notation}

\section{ Preliminaries}

\noindent
In this section, we shall list some lemmas which are necessary for establishing Theorem \ref{Theorem}.

\begin{lemma}\label{Finite-Fourier-expansion}
For any $H>1$, one has
\begin{equation*}
  \psi(\theta)=-\sum_{0<|h|\leqslant H}\frac{e(\theta h)}{2\pi ih}+O\left(g(\theta,H)\right),
\end{equation*}
where
\begin{equation*}
g(\theta,H)=\min\left(1,\frac{1}{H\|\theta\|}\right)
=\sum_{h=-\infty}^{\infty}a(h)e(\theta h),
\end{equation*}
\begin{equation*}
a(0)\ll \frac{\log 2H}{H},\ \
a(h)\ll \min\left(\frac{1}{|h|},\frac{H}{h^2}\right)\quad (h\neq 0).
\end{equation*}
\end{lemma}
\begin{proof}
See the arguments on page 245 of Heath--Brown \cite{Heath-Brown-1983}.
\end{proof}

\begin{lemma}\label{Davenport-Chapter25-1980}
Suppose that Suppose $\alpha$ is a real number and $a$ and $q$ are positive integers satisfying $(a,q)=1$
and $|\alpha-a/q|\leqslant q^{-2}$. Then, given any real number $\varepsilon>0$, one has
\begin{equation*}
\sum_{n\leqslant N}\Lambda(n)e(n^2\alpha)\ll N^{1+\varepsilon}(q^{-1}+N^{-1/2}+qN^{-2})^{1/4},
\end{equation*}
where the constant implied depends at most on $\varepsilon$.
\end{lemma}
\begin{proof}
See Theorem 2 of Ghosh \cite{Ghosh-1981}.
\end{proof}

\begin{lemma}\label{Karatsuba-1993}
Let
\begin{equation*}
\alpha=\frac{a}{q}+\frac{\theta}{q^2},\quad (a,q)=1,\quad q\geqslant1,\quad |\theta|\leqslant1.
\end{equation*}
Then, for any $\beta\in\mathbb{R}, \Delta\in\mathbb{R}, H\geqslant1$ and $N\geqslant1$, we have
\begin{equation*}
\sum_{n=1}^N\min\left(1,\frac{1}{H\|\alpha n^2+\beta\pm\Delta\|}\right)\ll\ (NHq)^{\varepsilon}\Big(Nq^{-1/2}+N^{1/2}+NH^{-1}+H^{-1/2}q^{1/2}\Big).
\end{equation*}
\end{lemma}
\begin{proof}
See the arguments on pp. 265--266 of Ghosh \cite{Ghosh-1981}.
\end{proof}

\begin{lemma}\label{3-derivatives}
Suppose that $f(x):[a; b]\to\mathbb{R}$ has continuous derivatives of arbitrary order on $[a,b]$, where
$1\leqslant a<b\leqslant2a$. Suppose further that
\begin{equation*}
\big|f^{(j)}(x)\big|\asymp\lambda_j,\qquad j\geqslant1,\qquad x\in[a,b].
\end{equation*}
Then we have
\begin{equation*}
\sum_{a<m\leqslant b}e(f(m))\ll a\lambda_3^{1/6}+\lambda_3^{-1/3}.
\end{equation*}
\begin{proof}
See Corollary 4.2 of Sargos \cite{Sargos-1995}.
\end{proof}
\end{lemma}

\begin{lemma}\label{k-min-esti}
Suppose that $1/2<\gamma_k<\dots<\gamma_1<1$ are fixed real numbers. Let
\begin{equation*}
S(M;\gamma_1,\dots,\gamma_k):=\sup_{(u_1,\dots,u_k)\in [0,1]^k}\sum_{M<m\leqslant2M}\prod_{j=1}^k\min\left(1,\frac{1}{H_j\|(m+u_j)^{\gamma_j}\|}\right),
\end{equation*}
where $H_j>1\,(j=1,\dots,k)$ are real numbers. If $\gamma_1+\dots+\gamma_k>k-\frac{1}{k+1}$, then
\begin{equation*}
S(M;\gamma_1,\dots,\gamma_k)\ll M(H_1\cdots H_k)^{-1}(\log M)^k+M^{\frac{k}{k+1}}(\log M)^{k}.
\end{equation*}
\end{lemma}
\begin{proof}
See Proposition $4$ of Zhai \cite{Zhai-1999}.
\end{proof}

Let $1/2<\gamma_2<\gamma_1<1$, $a\in\mathbb{R}$ and $b\in\mathbb{R}$ subject to $ab\neq0$, and $\alpha$ be a real number. $M$ and $M_1$ are sufficiently large numbers satisfying $M<M_1\leqslant2M$. Define
\begin{equation*}
S(M):=\sum_{M<m\leqslant M_1}e\big(\alpha m^2+am^{\gamma_1}+bm^{\gamma_2}\big),\qquad R=|a|M^{\gamma_1}+|b|M^{\gamma_2}.
\end{equation*}

\begin{lemma}\label{Li-Zhai-2022-improvement}
For $R\gg1$, the estimate
\begin{equation*}\label{ab>0}
S(M)\ll M^{1/2}R^{1/6}+MR^{-1/4}
\end{equation*}
holds uniformly for $\alpha\in\mathbb{R}$.
\end{lemma}
\begin{proof}
Let $f(m)=\alpha m^2+am^{\gamma_1}+bm^{\gamma_2}$. Then, for $m\in(M,M_1]$, we have
\begin{equation*}
f'''(m)=a\gamma_1(\gamma_1-1)(\gamma_1-2)m^{\gamma_1-3}+b\gamma_2(\gamma_2-1)(\gamma_2-2)m^{\gamma_2-3}.
\end{equation*}
If $ab>0$, then there holds $|f'''(m)|\asymp M^{-3}R$. By Lemma \ref{3-derivatives}, we get
\begin{equation}\label{ab>0-Sum-S(f(m))}
\sum_{M<m\leqslant M_1}e(f(m))\ll M^{1/2}R^{1/6}+MR^{-1/3}.
\end{equation}
Now, we suppose that $ab<0$. Define
\begin{equation*}
\mathcal{I}_0=\Big\{m:m\in(M,M_1],\,\big|f'''(m)\big|\leqslant R^{3/4}M^{-3}\Big\},
\end{equation*}
\begin{equation*}
\mathcal{I}_j=\Big\{m:m\in(M,M_1],\, 2^{j-1}R^{3/4}M^{-3}<\big|f'''(m)\big|\leqslant2^jR^{3/4}M^{-3}\leqslant2RM^{-3}\Big\},
\end{equation*}
where\begin{equation*}
    1\leqslant j\leqslant\left[\frac{\log2R^{1/4}}{\log 2}\right]=:J_0.
\end{equation*}
First, we give upper bound estimate of $|\mathcal{I}_0|$. For $m\in\mathcal{I}_0$, we deduce that
\begin{align*}
          a\gamma_1(\gamma_1-1)(\gamma_1-2)m^{\gamma_1-3}
= & \,\, -b\gamma_2(\gamma_2-1)(\gamma_2-2)m^{\gamma_2-3}+O(R^{3/4}M^{-3})
                 \nonumber \\
= & \,\, -b\gamma_2(\gamma_2-1)(\gamma_2-2)m^{\gamma_2-3}\big(1+O(R^{-1/4})\big),
\end{align*}
which implies that
\begin{align*}
m = & \,\, \bigg(-\frac{b\gamma_2(\gamma_2-1)(\gamma_2-2)}{a\gamma_1(\gamma_1-1)(\gamma_1-2)}\bigg)
           ^{\frac{1}{\gamma_1-\gamma_2}}\big(1+O(R^{-1/4})\big)^{\frac{1}{\gamma_1-\gamma_2}}
                  \nonumber \\
  = & \,\, \bigg(-\frac{b\gamma_2(\gamma_2-1)(\gamma_2-2)}{a\gamma_1(\gamma_1-1)(\gamma_1-2)}\bigg)
           ^{\frac{1}{\gamma_1-\gamma_2}}\big(1+O(R^{-1/4})\big)
                  \nonumber \\
  = & \,\, \bigg(-\frac{b\gamma_2(\gamma_2-1)(\gamma_2-2)}{a\gamma_1(\gamma_1-1)(\gamma_1-2)}\bigg)
           ^{\frac{1}{\gamma_1-\gamma_2}}+O\big(MR^{-1/4}\big).
\end{align*}
Therefore, one has
\begin{equation*}
|\mathcal{I}_0|\ll MR^{-1/4}.
\end{equation*}
Accordingly, it follows from Lemma \ref{3-derivatives} that
\begin{align*}
          \sum_{M<m\leqslant M_1}e(f(m))
 = & \,\, \sum_{m\in\mathcal{I}_0}e(f(m))+\sum_{1\leqslant j\leqslant J_0}\sum_{m\in\mathcal{I}_j}e(f(m))
                 \nonumber \\
\ll & \,\, MR^{-1/4}+\sum_{1\leqslant j\leqslant J_0}\Big(M(2^jR^{3/4}M^{-3})^{1/6}+(2^jR^{3/4}M^{-3})^{-1/3}\Big)
                 \nonumber \\
\ll & \,\, MR^{-1/4}+M^{1/2}R^{1/8}2^{J_0/6}
                 \nonumber \\
\ll & \,\, MR^{-1/4}+M^{1/2}R^{1/6},
\end{align*}
which completes the proof of Lemma \ref{Li-Zhai-2022-improvement}.
\end{proof}

\begin{lemma}\label{Heath-Brown-identity}
Let $z\geqslant 1$ and $k\geqslant 1$. Then, for any $n\leqslant 2z^k$, there holds
\begin{equation*}
\Lambda(n)=\sum_{j=1}^k(-1)^{j-1}\binom{k}{j}\mathop{\sum\cdots \sum}\limits_{\substack{n_1n_2\cdots n_{2j}=n\\n_{j+1,\dots,n_{2j}\leqslant z}}}(\log n_1)\mu(n_{j+1})\cdots \mu(n_{2j}).
\end{equation*}
\end{lemma}

\begin{proof}
See the arguments on pp. 1366--1367 of Heath--Brown \cite{Heath-Brown-1982}.
\end{proof}

\begin{lemma}\label{Weyl inequality}
Let $\mathcal{L},\mathcal{Q}\geqslant1$ and $z_\ell$ be complex numbers. Then there holds
\begin{equation*}
  \left|\sum_{\mathcal{L}<\ell\leqslant 2\mathcal{L}}z_\ell\right|^2\leqslant\bigg(2+\frac{\mathcal{L}}{\mathcal{Q}}\bigg)
  \sum_{0\leqslant|q|\leqslant \mathcal{Q}}\left(1-\frac{|q|}{\mathcal{Q}}\right)
  \sum_{\mathcal{L}<\ell+q,\ell-q\leqslant 2\mathcal{L}}z_{\ell+q}\overline{z_{\ell-q}}.
\end{equation*}
\end{lemma}
\begin{proof}
See Lemma 2 of Fouvry and Iwaniec \cite{Fouvry-Iwaniec-1989}.
\end{proof}

\section{ Exponential Sum Estimate over Primes}
Let $1/2<\gamma_2<\gamma_1<1$, $\gamma_1+\gamma_2>27/14$, $m\sim M$, $k\sim K$, $mk\sim N$, and $h_1,h_2$ be integers satisfying $1\leqslant|h_1|\leqslant N^{1-\gamma_1+\frac{14(\gamma_1+\gamma_2)-27}{43}-\frac{\varepsilon}{3}}$ and
$1\leqslant|h_2|\leqslant N^{1-\gamma_2+\frac{14(\gamma_1+\gamma_2)-27}{43}-\frac{\varepsilon}{3}}$.
For any $t\in\mathbb{R}$, define
\begin{align*}
& S_I(M,K):=\sum_{M<m\leqslant2M}a(m)\sum_{K<k\leqslant2K}e(\alpha tm^2k^2+h_1m^{\gamma_1}k^{\gamma_1}+h_2m^{\gamma_2}k^{\gamma_2}), \\
& S_{II}(M,K):=\sum_{M<m\leqslant2M}a(m)\sum_{K<k\leqslant2K}b(k)e(\alpha tm^2k^2+h_1m^{\gamma_1}k^{\gamma_1}+h_2m^{\gamma_2}k^{\gamma_2}),
\end{align*}
where $a(m)$ and $b(k)$ are complex numbers satisfying $a(m)\ll1$, $b(k)\ll1$. For convenience, we write $R_*=|h_1|N^{\gamma_1}+|h_2|N^{\gamma_2}$. Trivially, there holds $N^{\gamma_1}\ll R_*\ll N^{\frac{14(\gamma_1+\gamma_2)+16}{43}-\frac{\varepsilon}{3}}$.

\begin{lemma}\label{S(II)}
Suppose that $a(m)\ll1$, $b(k)\ll1$. Then, for $N^{1/4}\ll K\ll N^{1/2}$, we have
\begin{equation*}
S_{II}(M,K)\ll N^{\frac{38+(\gamma_1+\gamma_2)}{43}+\varepsilon}.
\end{equation*}
\end{lemma}
\begin{proof}
Let $Q=\big\lfloor N^{\frac{10-2(\gamma_1+\gamma_2)}{43}-2\varepsilon}\big\rfloor=o(K)$. By Cauchy's inequality and Lemma \ref{Weyl inequality}, we have
\begin{align*}
           |S_{II}(M,K)|^2
\ll & \,\, M\sum_{M<m\leqslant2M}\Bigg|\sum_{K<k\leqslant2K}b(k)
           e(\alpha tm^2k^2+h_1m^{\gamma_1}k^{\gamma_1}+h_2m^{\gamma_2}k^{\gamma_2})\Bigg|^2
                   \nonumber \\
\ll & \,\, M\sum_{M<m\leqslant2M}\frac{K}{Q}\sum_{0\leqslant|q|\leqslant Q}
           \bigg(1-\frac{|q|}{Q}\bigg)\sum_{\substack{K<k+q\leqslant2K\\ K<k-q\leqslant 2K}}
           b(k+q)\overline{b(k-q)}e(f(m,k;q))
                   \nonumber \\
\ll & \,\, \frac{M^2K^2}{Q}+\frac{MK}{Q}\sum_{1\leqslant |q|\leqslant Q}
           \sum_{K<k\leqslant 2K}\Bigg|\sum_{M<m\leqslant2M}e(f(m,k;q))\Bigg|,
\end{align*}
where
\begin{align*}
& f(m,k;q)=4\alpha tm^2kq+h_1m^{\gamma_1}\cdot\Delta(k,q;\gamma_1)+h_2m^{\gamma_2}\cdot\Delta(k,q;\gamma_2),\\
& \Delta(k,q;\gamma_i)=(k+q)^{\gamma_i}-(k-q)^{\gamma_i},\qquad(i=1,2).
\end{align*}
By an elementary asymptotic
\begin{equation*}
\Delta(k,q;\gamma)=2\gamma qk^{\gamma-1}+O(q^2K^{\gamma-2})\asymp|q|K^{\gamma-1},
\end{equation*}
we deduce that
\begin{align*}
R = & \,\, |h_1\Delta(k,q;\gamma_1)|M^{\gamma_1}+|h_2\Delta(k,q;\gamma_2)|M^{\gamma_2}
           \asymp |h_1||q|K^{\gamma_1-1}M^{\gamma_1}+|h_2||q|K^{\gamma_2-1}M^{\gamma_2}
                     \nonumber \\
\asymp & \,\, |q|K^{-1}(|h_1|N^{\gamma_1}+|h_2|N^{\gamma_2})=|q|K^{-1}R_* \gg 1,
\end{align*}
which combined with Lemma \ref{Li-Zhai-2022-improvement} with parameters
$(\alpha,a,b)=(2\alpha tkq,h_1\Delta(k,q;\gamma_1),h_2\Delta(k,q;\gamma_2))$ yields
\begin{align*}
    \sum_{1\leqslant|q|\leqslant Q}\sum_{K<k\leqslant2K}
  & \Big(M^{1/2}\big(|q|K^{-1}R_*\big)^{1/6}+M\big(|q|K^{-1}R_*\big)^{-1/4}\Big)
            \nonumber \\
  &\ll M^{1/2}K^{5/6}Q^{7/6}R^{1/6}_*+MK^{5/4}Q^{3/4}R_*^{-1/4}\ll N.
\end{align*}
This completes the proof of Lemma \ref{S(II)}.
\end{proof}

\begin{lemma}\label{S(I)}
Suppose that $a(m)\ll 1$. Then, for $M\ll N^{1/4}$, we have
\begin{equation*}
S_{I}(M,K)\ll N^{\frac{38+(\gamma_1+\gamma_2)}{43}+\varepsilon}.
\end{equation*}
\end{lemma}
\begin{proof}
Taking $(\alpha,a,b)=(\alpha tm^2, h_1m^{\gamma_1}, h_2m^{\gamma_2})$ in Lemma \ref{Li-Zhai-2022-improvement}, it is easy to see that $R\asymp|h_1|M^{\gamma_1}K^{\gamma_1}+|h_2|M^{\gamma_2}K^{\gamma_2}
\asymp|h_1|N^{\gamma_1}+|h_2|N^{\gamma_2}=R_*$. Then, by Lemma \ref{Li-Zhai-2022-improvement}, we get
\begin{align*}
S_{I}(M,K)
&\ll\sum_{M<m\leqslant2M}\Bigg|\sum_{K<k\leqslant 2K}e(\alpha tm^2k^2+h_1m^{\gamma_1}k^{\gamma_1}+h_2m^{\gamma_2}k^{\gamma_2})\Bigg|\\
&\ll\sum_{M<m\leqslant2M}\big(K^{1/2}R^{1/6}_*+KR^{-1/4}_*\big)\ll MK^{1/2}R^{1/6}_*+NR^{-1/4}_* \ll N^{\frac{38+(\gamma_1+\gamma_2)}{43}+\varepsilon}.
\end{align*}
This completes the proof of Lemma \ref{S(I)}.
\end{proof}

\begin{lemma}\label{Applying-H-B-identity}
Suppose that $N/2<n\leqslant N$. Then we have
\begin{equation*}
\Gamma^*(N):=\sum_{N/2<n\leqslant N}\Lambda(n)e(\alpha tn^2+h_1n^{\gamma_1}+h_2n^{\gamma_2})\ll N^{\frac{38+(\gamma_1+\gamma_2)}{43}+\frac{7\varepsilon}{6}}.
\end{equation*}
\end{lemma}
\begin{proof}
By Heath--Brown's identity, i.e., Lemma \ref{Heath-Brown-identity}, with $k=3$, one can see that $\Gamma^*(N)$ can be written as linear combination of $O\big((\log N)^6\big)$ sums, each of which is of the form
\begin{equation}\label{Gamma-(N)}
\Gamma^{\dag}(N):=\sum_{n_1\sim N_1}\!\!\!\cdots\!\!\!\sum_{n_6\sim N_6}(\log n_1)\mu(n_4)\mu(n_5)\mu(n_6)
e\big(\alpha t(n_1\cdots n_6)^2+h_1(n_1\cdots n_6)^{\gamma_1}+h_2(n_1\cdots n_6)^{\gamma_2}\big),
\end{equation}
where $N\ll N_1\cdots N_6\ll N$; $2N_i\leqslant(2N)^{1/3}$, $i=4,5,6$ and some $n_i$ may only take value $1$. Therefore, it is sufficient to give upper bound estimate for each $\Gamma^{\dag}(N)$ defined as in (\ref{Gamma-(N)}). Next, we shall consider four cases.

\noindent
\textbf{Case 1.} If there exists an $N_j$ such that $N_j\geqslant N^{3/4}>N^{1/3}$, then we must have $1\leqslant j\leqslant 3$. Without loss of generality, we postulate that $N_1\gg N^{3/4}$, and take $m=n_2n_3\cdots n_6,\,\,k=n_1$. Trivially, there holds $m\ll N^{1/4}$. Set
\begin{equation*}
a(m)=\sum_{m=n_2n_3\cdots n_6}\mu(n_4)\cdots\mu(n_6)\ll \tau_5(m).
\end{equation*}
Then $\Gamma^{\dag}(X)$ is a sum of the form $S_I(M,K)$. By Lemma \ref{S(I)}, we have
\begin{equation*}
N^{-\frac{\varepsilon}{10}}\cdot\Gamma^{\dag}(N)\ll N^{\frac{38+(\gamma_1+\gamma_2)}{43}+\varepsilon}.
\end{equation*}

\noindent
\textbf{Case 2.}
If there exists an $N_j$ such that $N^{1/4}\leqslant N_j\leqslant N^{1/2}$,
then we take
\begin{equation*}
k=n_j,\quad K=N_j,\qquad m=\prod_{\substack{1\leqslant i\leqslant6\\ i\neq j}}n_i,\quad
M=\prod_{\substack{1\leqslant i\leqslant6\\ i\neq j}}N_i.
\end{equation*}
Thus, $\Gamma^{\dag}(N)$ is a sum of the form $S_{II}(M,K)$ with $N^{1/4}\ll K\ll N^{1/2}$. By Lemma \ref{S(II)}, we have
\begin{equation*}
N^{-\frac{\varepsilon}{10}}\cdot\Gamma^{\dag}(N)\ll N^{\frac{38+(\gamma_1+\gamma_2)}{43}+\varepsilon}.
\end{equation*}

\noindent
\textbf{Case 3.}
If there exists an $N_j$ such that $N^{1/2}< N_j< N^{3/4}$,
then we take
\begin{equation*}
m=n_j,\quad M=N_j,\qquad k=\prod_{\substack{1\leqslant i\leqslant6\\ i\neq j}}n_i,
\quad K=\prod_{\substack{1\leqslant i\leqslant6\\ i\neq j}}N_i.
\end{equation*}
Thus, $\Gamma^{\dag}(N)$ is a sum of the form $S_{II}(M,K)$ with $N^{1/4}\ll K\ll N^{1/2}$. By Lemma \ref{S(II)}, we have
\begin{equation*}
N^{-\frac{\varepsilon}{10}}\cdot\Gamma^{\dag}(N)\ll N^{\frac{38+(\gamma_1+\gamma_2)}{43}+\varepsilon}.
\end{equation*}

\noindent
\textbf{Case 4.} If $N_j< N^{1/4}\,(j=1,2,\dots,6)$, without loss of generality, we postulate that $N_1\geqslant N_2\geqslant \dots\geqslant N_6$. Let $\ell$ be the natural number such that
\begin{equation*}
N_1\cdots N_{\ell-1}< N^{1/4}, \qquad N_1\cdots N_{\ell}\geqslant N^{1/4}.
\end{equation*}
It is easy to check that $2\leqslant \ell\leqslant5$. Then we have
\begin{equation*}
N^{1/4}\leqslant N_1\cdots N_{\ell}=N_1\cdots N_{\ell-1}N_{\ell}< N^{1/4}\cdot N^{1/4}\leqslant N^{1/2}.
\end{equation*}
In this case, we take
\begin{equation*}
m=\prod_{j=\ell+1}^6 n_j,\quad M=\prod_{j=\ell+1}^6 N_j,\qquad k=\prod_{j=1}^{\ell}n_j,\quad K=\prod_{j=1}^{\ell}N_j.
\end{equation*}
Then $\Gamma^{\dag}(N)$ is a sum of the form $S_{II}(M,K)$. By Lemma \ref{S(II)}, we have
\begin{equation*}
N^{-\frac{\varepsilon}{10}}\cdot\Gamma^{\dag}(N)\ll N^{\frac{38+(\gamma_1+\gamma_2)}{43}+\varepsilon}.
\end{equation*}
Combining the above four cases, we derive that
\begin{equation*}
\Gamma^{*}(N)\ll \Gamma^{\dag}(N)\cdot(\log N)^6\ll N^{\frac{38+(\gamma_1+\gamma_2)}{43}+\frac{7\varepsilon}{6}}.
\end{equation*}
This completes the proof of Lemma \ref{Applying-H-B-identity}.
\end{proof}

\section{Proof of Theorem \ref{Theorem}}

Define a periodic function $\mathcal{F}_{\Delta}(\theta)$ with period $1$ such that
\begin{align*}
\mathcal{F}_{\Delta}(\theta)=
\begin{cases}
     0,  & \textrm{if $-1/2\leqslant \theta\leqslant-\Delta$},\\
     1,  & \textrm{if $-\Delta<\theta<\Delta$},\\
     0,  & \textrm{if $\Delta\leqslant \theta\leqslant 1/2$}.
\end{cases}
\end{align*}
\begin{lemma}\label{upper-bound-Upsilon}
Let
\begin{equation*}
\Upsilon(\gamma_1,\gamma_2;N):=\sum_{\substack{p\leqslant N\\p=\lfloor n^{1/\gamma_1}_1\rfloor=\lfloor n^{1/\gamma_2}_2\rfloor}}\big(\mathcal{F}_{\Delta}(\alpha p^2+\beta)-2\Delta\big)\log p
\end{equation*}
with $\Delta=N^{-\frac{14(\gamma_1+\gamma_2)-27}{43}+\varepsilon}$. Then we have
\begin{align*}
\Upsilon(\gamma_1,\gamma_2;N)\ll N^{\frac{29(\gamma_1+\gamma_2)-16}{43}+\frac{5\varepsilon}{6}}.
\end{align*}
\end{lemma}
\noindent
From Lemma \ref{upper-bound-Upsilon}, we know that
\begin{equation}\label{asymptotic-F}
\sum_{\substack{p\leqslant N\\p=\lfloor n^{1/\gamma_1}_1\rfloor=\lfloor n^{1/\gamma_2}_2\rfloor}}\mathcal{F}_{\Delta}(\alpha p^2+\beta)\log p
=2\Delta\sum_{\substack{p\leqslant N\\p=\lfloor n^{1/\gamma_1}_1\rfloor=\lfloor n^{1/\gamma_2}_2\rfloor}}\log p+O\Big(N^{\frac{29(\gamma_1+\gamma_2)-16}{43}+\frac{5\varepsilon}{6}}\Big).
\end{equation}
By (\ref{asymptotic-F}), (\ref{2-P-S-PNT}) and the definition of $\Delta$, we deduce that
\begin{equation*}
 \sum_{\substack{p\leqslant N\\p=\lfloor n^{1/\gamma_1}_1\rfloor=\lfloor n^{1/\gamma_2}_2\rfloor}}\mathcal{F}_{\Delta}(\alpha p^2+\beta)\log p\gg
 N^{\frac{29(\gamma_1+\gamma_2)-16}{43}+\varepsilon},
\end{equation*}
which implies the conclusion of Theorem \ref{Theorem} thanks to the definition of $\mathcal{F}_{\Delta}(\theta)$. Therefore, in the rest of this section, we shall prove Lemma \ref{upper-bound-Upsilon}. For $1/2<\gamma<1$, it is easy to see that
\begin{equation*}
\lfloor-p^{\gamma}\rfloor-\lfloor-(p+1)^{\gamma}\rfloor=
\begin{cases}
     1,  & \textrm{if $p=\lfloor n^{1/\gamma}\rfloor$ for some $n\in\mathbb{N}^+$},\\
     0, & \textrm{otherwise},
   \end{cases}
\end{equation*}
and
\begin{equation}\label{Elementary-Formula}
(p+1)^{\gamma}-p^{\gamma}=\gamma p^{\gamma-1}+O(p^{\gamma-2}).
\end{equation}
Thus, we have
\begin{align*}
         \Upsilon(\gamma_1,\gamma_2;N)
= &  \sum_{p\leqslant N}\big(\lfloor-p^{\gamma_1}\rfloor-\lfloor-(p+1)^{\gamma_1}\rfloor\big)
         \big(\lfloor -p^{\gamma_2}\rfloor-\lfloor-(p+1)^{\gamma_2}\rfloor\big)
         \big(\mathcal{F}_{\Delta}(\alpha p^2+\beta)-2\Delta\big)\log p
                 \nonumber \\
= &  \sum_{p\leqslant N}\Bigg(\prod_{i=1}^2\big((p+1)^{\gamma_i}-p^{\gamma_i}+\psi(-(p+1)^{\gamma_i})
         -\psi(-p^{\gamma_i})\big)\Bigg)\big(\mathcal{F}_{\Delta}(\alpha p^2+\beta)-2\Delta\big)\log p
                 \nonumber \\
=: & \,\,\Upsilon_1(\gamma_1,\gamma_2;N)+\Upsilon_2(\gamma_1,\gamma_2;N)+\Upsilon_3(\gamma_1,\gamma_2;N)
          +\Upsilon_4(\gamma_1,\gamma_2;N),
\end{align*}
where
\begin{align*}
&  \Upsilon_1(\gamma_1,\gamma_2;N)=\sum_{p\leqslant N}\big((p+1)^{\gamma_1}-p^{\gamma_1}\big)
   \big((p+1)^{\gamma_2}-p^{\gamma_2}\big)\big(\mathcal{F}_{\Delta}(\alpha p^2+\beta)-2\Delta\big)\log p,
              \nonumber \\
&  \Upsilon_2(\gamma_1,\gamma_2;N)=\sum_{p\leqslant N}\big((p+1)^{\gamma_1}-p^{\gamma_1}\big)
   \big(\psi(-(p+1)^{\gamma_2})-\psi(-p^{\gamma_2})\big)
   \big(\mathcal{F}_{\Delta}(\alpha p^2+\beta)-2\Delta\big)\log p,
              \nonumber \\
&  \Upsilon_3(\gamma_1,\gamma_2;N)=\sum_{p\leqslant N}\big((p+1)^{\gamma_2}-p^{\gamma_2}\big)
   \big(\psi(-(p+1)^{\gamma_1})-\psi(-p^{\gamma_1})\big)
   \big(\mathcal{F}_{\Delta}(\alpha p^2+\beta)-2\Delta\big)\log p,
              \nonumber \\
&  \Upsilon_4(\gamma_1,\gamma_2;N)=\sum_{p\leqslant N}
   \Bigg(\prod_{i=1}^2\big(\psi(-(p+1)^{\gamma_i})-\psi(-p^{\gamma_i})\big)\Bigg)
   \big(\mathcal{F}_{\Delta}(\alpha p^2+\beta)-2\Delta\big)\log p.
\end{align*}
Now, we use the well--known expansion (e.g., see the arguments on page 140 of Vaughan \cite{Vaughan-1977})
\begin{align}\label{Expansion-F-Delta}
\mathcal{F}_{\Delta}(\theta)-2\Delta
& =\sum_{1\leqslant|t|\leqslant T}\frac{\sin2\pi t\Delta}{\pi t}e(t\theta)+O\Bigg(\min\bigg(1,\frac{1}{T\|\theta+\Delta\|}\bigg)
+\min\bigg(1,\frac{1}{T\|\theta-\Delta\|}\bigg)\Bigg)\nonumber\\
& =\mathcal{M}(\theta,T)+O\big(\mathcal{E}(\theta,T)\big),
\end{align}
say. For $\Upsilon_1(\gamma_1,\gamma_2;N)$, by a splitting argument, it suffices to estimate
\begin{equation*}
\Upsilon^*_1(\gamma_1,\gamma_2;N)
:=\sum_{N/2<p\leqslant N}\big((p+1)^{\gamma_1}-p^{\gamma_1}\big)\big((p+1)^{\gamma_2}-p^{\gamma_2}\big)\big(\mathcal{F}_{\Delta}(\alpha p^2+\beta)-2\Delta\big)\log p.
\end{equation*}
Putting (\ref{Expansion-F-Delta}) into the right--hand side of $\Upsilon^*_1(\gamma_1,\gamma_2;N)$, it follows from (\ref{Elementary-Formula}) that
\begin{align}\label{upper-bound-Upsilon-star-1}
& \Upsilon^*_1(\gamma_1,\gamma_2;N)\nonumber\\
= & \sum_{N/2<p\leqslant N}\big(\gamma_1\gamma_2p^{\gamma_1+\gamma_2-2}+O(p^{\gamma_1+\gamma_2-3})\big)
\big(\mathcal{M}({\alpha p^2+\beta},T)+O(\mathcal{E}(\alpha p^2+\beta,T))\big)\log p\nonumber\\
\ll & \sum_{N/2<p\leqslant N}p^{\gamma_1+\gamma_2-2}\cdot\mathcal{M}({\alpha p^2+\beta},T)\log p+\sum_{N/2<p\leqslant N}p^{\gamma_1+\gamma_2-2}\cdot\big|\mathcal{E}(\alpha p^2+\beta,T)\big|\log p.
\end{align}
By noting that $1/2<\gamma_2<\gamma_1<1$ and $27/14<\gamma_1+\gamma_2<2$, one has $\gamma_2>27/14-\gamma_1>27/14-1=13/14$. Taking $T=\lfloor q^{1/2}\rfloor$ and $q=N^{\frac{12-6\gamma_2}{13}}$ with $13/14<\gamma_2<1$, by Lemma \ref{Karatsuba-1993}, the second term on the right--hand side of (\ref{upper-bound-Upsilon-star-1}) can be estimated as
\begin{align}\label{error-term}
\ll & \,\, \log N  \sum_{N/2<n\leqslant N}\Bigg(\min\bigg(1,\frac{1}{T\|\alpha n^2+\beta+\Delta\|}\bigg)
           +\min\bigg(1,\frac{1}{T\|\alpha n^2+\beta-\Delta\|}\bigg)\Bigg)
                       \nonumber\\
\ll & \,\, (NTq)^{\frac{\varepsilon}{10}}\big(Nq^{-1/2}+N^{1/2}+NT^{-1}+T^{-1/2}q^{1/2}\big)\ll  N^{\varepsilon}\big(N q^{-1/2}+N^{1/2}+q^{1/4}\big)                     \nonumber\\
\ll & \,\, \big(N^{\frac{3\gamma_2+7}{13}}+N^{1/2}+N^{\frac{6-3\gamma_2}{26}}\big)N^{\varepsilon}
           \ll N^{\frac{29(\gamma_1+\gamma_2)-16}{43}}.
\end{align}
\noindent
Hence, it is sufficient to show that
\begin{equation*}
\Xi(\gamma_1,\gamma_2,N):=\sum_{N/2<p\leqslant N}(\log p)p^{\gamma_1+\gamma_2-2}\cdot\mathcal{M}(\alpha p^2+\beta,T)\ll N^{\frac{29(\gamma_1+\gamma_2)-16}{43}+\frac{2\varepsilon}{3}}.
\end{equation*}
By partial summation, one has
\begin{align}\label{Xi}
           \Xi(\gamma_1,\gamma_2;N)
  = & \,\, \sum_{N/2<p\leqslant N}p^{\gamma_1+\gamma_2-2}\sum_{1\leqslant|t|\leqslant T}
           \frac{\sin2\pi t\Delta}{\pi t}e(\alpha tp^2+\beta t)\log p
                 \nonumber\\
\ll & \,\, \sum_{1\leqslant|t|\leqslant T}\bigg|\frac{\sin 2\pi t\Delta}{\pi t}\bigg|
           \cdot\Bigg|\sum_{N/2<p\leqslant N}p^{\gamma_1+\gamma_2-2}e(\alpha tp^2)\log p\Bigg|
                 \nonumber\\
\ll & \,\, \sum_{1\leqslant t\leqslant T}\min\bigg(\Delta,\frac{1}{t}\bigg)
           \Bigg|\sum_{N/2<p\leqslant N}p^{\gamma_1+\gamma_2-2}e(\alpha tp^2)\log p\Bigg|
                 \nonumber\\
\ll & \,\, \sum_{1\leqslant t\leqslant T}\min\bigg(\Delta,\frac{1}{t}\bigg)
           \Bigg|\int_{\frac{N}{2}}^{N}u^{\gamma_1+\gamma_2-2}\mathrm{d}
           \Bigg(\sum_{N/2<p\leqslant u}e(\alpha tp^2)\log p\Bigg)\Bigg|
                 \nonumber\\
\ll & \,\, \sum_{1\leqslant t\leqslant T}\min\bigg(\Delta,\frac{1}{t}\bigg)
           \Bigg(N^{\gamma_1+\gamma_2-2}\Bigg|\sum_{N/2<p\leqslant N}e(\alpha tp^2)\log p\Bigg|
                 \nonumber\\
    & \,   \quad+\max_{N/2\leqslant v\leqslant N}\Bigg|\sum_{N/2<p\leqslant v}e(\alpha tp^2)\log p\Bigg|
           \times\int_{\frac{N}{2}}^{N}u^{\gamma_1+\gamma_2-3}\mathrm{d}u\Bigg)
                 \nonumber\\
\ll & \,\, N^{\gamma_1+\gamma_2-2}\cdot
           \sum_{1\leqslant t \leqslant T}\min\bigg(\Delta,\frac{1}{t}\bigg)\cdot
           \max_{N/2\leqslant u\leqslant N}\Bigg|\sum_{N/2<p\leqslant u}e(\alpha tp^2)\log p\Bigg|.
\end{align}
Since $\alpha$ is irrational, there exist infinitely many distinct convergents  $a/q$ to its continued fraction subject to
\begin{equation}\label{approximation-a}
  \bigg|\alpha-\frac{a}{q}\bigg|\leqslant\frac{1}{q^2},\qquad (a,q)=1,\qquad q\geqslant1.
\end{equation}
By Dirichlet's lemma on rational approximation (e.g., see Lemma 2.1 of Vaughan \cite{Vaughan-book}),
for each $1\leqslant t\leqslant T$, there exist integers $b_t$ and $r_t$ such that
\begin{equation}\label{approximation-alpha t}
\left|\alpha t-\frac{b_t}{r_t}\right|\leqslant \frac{1}{r_tq^2},\qquad (b_t,r_t)=1,
\qquad 1\leqslant r_t\leqslant q^2.
\end{equation}
We claim that, uniformly for $1\leqslant t\leqslant T$, there holds
\begin{equation}\label{q-upp-low}
q^{1/3}<r_t\leqslant q^2.
\end{equation}
Otherwise, if $r_t\leqslant q^{1/3}$, one has
\begin{equation}\label{tr-q}
tr_t\leqslant T\cdot q^{1/3}\leqslant q^{1/2}\cdot q^{1/3}= q^{5/6}.
\end{equation}
Now, we shall illustrate
\begin{equation}\label{aq-neq-btr}
\frac{a}{q}\neq\frac{b_t}{tr_t}.
\end{equation}
Otherwise, one has $atr_t=b_tq$. Since $(a,q)=(b_t,r_t)=1$, then $r_t|q$ and $a|b_t$. Set $b_t=a\cdot k_t$, then
$tr_t=k_tq$. It follows from $(b_t,r_t)=1$ that $(k_t,r_t)=1$, which implies $k_t|t$. By setting $t=k_t\cdot\ell_t$, we get $q=\ell_t\cdot r_t\leqslant tr_t$, which contradicts to (\ref{tr-q}). Hence, (\ref{aq-neq-btr}) holds. It follows from (\ref{tr-q}) and (\ref{aq-neq-btr}) that
\begin{equation}\label{approximate-at}
\bigg|\frac{a}{q}-\frac{b_t}{tr_t}\bigg|=\frac{|atr_t-b_tq|}{tr_tq}\geqslant\frac{1}{tr_tq}
\geqslant\frac{1}{q^{11/6}}.
\end{equation}
On the other hand, by (\ref{approximation-a}) and (\ref{approximation-alpha t}), we deduce that
\begin{equation*}
\bigg|\frac{a}{q}-\frac{b_t}{tr_t}\bigg|\leqslant\bigg|\alpha-\frac{a}{q}\bigg|
+\bigg|\alpha-\frac{b_t}{tr_t}\bigg|\leqslant\frac{1}{q^2}+\frac{1}{tr_tq^2}\leqslant \frac{2}{q^2}
=o\bigg(\frac{1}{q^{11/6}}\bigg),
\end{equation*}
which contradicts to (\ref{approximate-at}). By Lemma \ref{Davenport-Chapter25-1980}, the innermost summation over $p$ of (\ref{Xi}) can be
estimated as
\begin{equation*}
\max_{N/2\leqslant u\leqslant N}\Bigg|\sum_{N/2<p\leqslant u}e(\alpha tp^2)\log p\Bigg|\ll N^{1+\varepsilon}\big(r_t^{-1}+N^{-1/2}+r_tN^{-2}\big)^{1/4},
\end{equation*}
which combined with (\ref{q-upp-low}) yields
\begin{align*}
           \Xi(\gamma_1,\gamma_2;N)
\ll & \,\, N^{\gamma_1+\gamma_2-1+\varepsilon}(q^{-1/12}+N^{-1/8}+N^{-1/2}q^{1/2})
           \cdot\sum_{1\leqslant t\leqslant T}\min\Big(\Delta,\frac{1}{t}\Big)
                 \nonumber \\
\ll & \,\, \Big(N^{\gamma_1+\gamma_2-1}\big(N^{\frac{12-6\gamma_2}{13}}\big)^{-1/12}+N^{\gamma_1+\gamma_2-9/8}
           +N^{\gamma_1+\gamma_2-3/2}N^{\frac{6-3\gamma_2}{13}}\Big)N^{\varepsilon}
                 \nonumber \\
\ll & \,\, N^{\frac{29(\gamma_1+\gamma_2)-16}{43}+\frac{2\varepsilon}{3}}.
\end{align*}

By (\ref{Elementary-Formula}), partial summation and the arguments on pp. 9--11 of Dimitrov \cite{Dimitrov-2025}, we know that, for $13/14<\gamma_2<1$, there holds
\begin{align*}
    & \,\, \Upsilon_2(\gamma_1,\gamma_2;N)
                \nonumber \\
\ll & \,\, \sum_{N/2<p\leqslant N}\big(\gamma_1p^{\gamma_1-1}+O(p^{\gamma_1-2})\big)
           \big(\psi(-(p+1)^{\gamma_2})-\psi(p^{\gamma_2})\big)
           (\mathcal{F}_{\Delta}(\alpha p^2+\beta)-2\Delta\big)\log p
                \nonumber \\
\ll & \,\, \sum_{N/2<p\leqslant N}p^{\gamma_1-1}\big(\psi(-(p+1)^{\gamma_2})-\psi(p^{\gamma_2})\big)
           (\mathcal{F}_{\Delta}(\alpha p^2+\beta)-2\Delta\big)\log p+N^{\gamma_1-1}
                \nonumber \\
\ll & \,\, N^{\gamma_1-1}(\log N)\max_{N/2<u\leqslant N}\Bigg|\sum_{N/2<p\leqslant u}
           \big(\psi(-(p+1)^{\gamma_2})-\psi(p^{\gamma_2})\big)
           (\mathcal{F}_{\Delta}(\alpha p^2+\beta)-2\Delta\big)\log p\Bigg|
                \nonumber \\
\ll & \,\, N^{\gamma_1-1}(\log N)\cdot N^{\frac{15\gamma_2+13}{29}+\frac{2\varepsilon}{3}}
           \ll N^{\frac{29\gamma_1+15\gamma_2-16}{29}+\frac{2\varepsilon}{3}}
           \ll N^{\frac{29(\gamma_1+\gamma_2)-16}{43}+\frac{2\varepsilon}{3}}.
\end{align*}
Noting that $1/2<\gamma_2<\gamma_1<1$ and $27/14<\gamma_1+\gamma_2<2$, one has $\gamma_1\in(27/28,1)\subseteq(13/14,1)$. By following the processes exactly the same as those of the upper bound
estimate of $\Upsilon_2(\gamma_1,\gamma_2;N)$, we can also derive that
$\Upsilon_3(\gamma_1,\gamma_2;N)\ll N^{\frac{29(\gamma_1+\gamma_2)-16}{43}+\frac{2\varepsilon}{3}}$.

Next, we focus on the upper bound estimate of $\Upsilon_4(\gamma_1,\gamma_2;N)$. Combining (\ref{Expansion-F-Delta}) and the arguments exactly the same as (\ref{error-term}), we obtain
\begin{align*}
     &   \Upsilon_4(\gamma_1,\gamma_2;N)
                  \nonumber \\
 \ll & \sum_{N/2<p\leqslant N}\Bigg(\prod_{i=1}^2\big(\psi(-(p+1)^{\gamma_i})-\psi(-p^{\gamma_i})\big)\Bigg)
       \cdot \mathcal{M}(\alpha p^2+\beta,T)\log p+N^{\frac{29(\gamma_1+\gamma_2)-16}{43}}
                  \nonumber \\
 \ll & \sum_{1\leqslant|t|\leqslant T}\frac{\sin2\pi t\Delta}{\pi t}e(\beta t)\sum_{N/2<p\leqslant N}
       \Bigg(\prod_{i=1}^2\big(\psi(-(p+1)^{\gamma_i})-\psi(-p^{\gamma_i})\big)\Bigg)e(\alpha tp^2)
       \log p+N^{\frac{29(\gamma_1+\gamma_2)-16}{43}}
                  \nonumber \\
\ll & \sum_{1\leqslant t\leqslant T}\min\bigg(\Delta,\frac{1}{t}\bigg)\Bigg|\sum_{N/2<p\leqslant N}
      \Bigg(\prod_{i=1}^2\big(\psi(-(p+1)^{\gamma_i})-\psi(-p^{\gamma_i})\big)\Bigg)e(\alpha tp^2)\log p\Bigg|
      +N^{\frac{29(\gamma_1+\gamma_2)-16}{43}}.
\end{align*}
Therefore, it is sufficient to show, uniformly for $1\leqslant t\leqslant T$, that
\begin{equation*}
\Gamma(N):=\sum_{N/2<n\leqslant N}
\Lambda(n)\Bigg(\prod_{i=1}^2\big(\psi(-(n+1)^{\gamma_1})-\psi(-n^{\gamma_1})\big)\Bigg) e(\alpha tn^2)
\ll N^{\frac{29(\gamma_1+\gamma_2)-16}{43}+\frac{\varepsilon}{2}}.
\end{equation*}
Let $H_1=N^{1-\gamma_1+\frac{14(\gamma_1+\gamma_2)-27}{43}-\frac{\varepsilon}{3}},\,
H_2=N^{1-\gamma_2+\frac{14(\gamma_1+\gamma_2)-27}{43}-\frac{\varepsilon}{3}}$. From Lemma \ref{Finite-Fourier-expansion},
we have
\begin{equation}\label{psi-psi}
\psi\left(-(n+1)^{\gamma_i}\right)-\psi\left(-n^{\gamma_i}\right)=M_{H_i}(n)+E_{H_i}(n),\qquad (i=1,2),
\end{equation}
where
\begin{align}
& M_{H_i}(n)=-\sum_{1\leqslant |h_i|\leqslant H_i}\frac{e\left(-h_i(n+1)^{\gamma_i}\right)-e\left(-h_in^{\gamma_i}\right)}{2\pi ih_i},  \label{MH(n)} \\
& E_{H_i}(n)=O\bigg(\min\bigg(1,\frac{1}{H_i\|(n+1)^{\gamma_i}\|}\bigg)\bigg)
+O\bigg(\min\bigg(1,\frac{1}{H_i\|n^{\gamma_i}\|}\bigg)\bigg). \label{EH(n)}
\end{align}
Inserting (\ref{psi-psi}) into $\Gamma(N)$, we obtain
\begin{equation*}
\Gamma(N)
=\sum_{N/2<n\leqslant N}\Lambda(n)\big(M_{H_1}(n)+E_{H_1}(n)\big)\big(M_{H_2}(n) +E_{H_2}(n)\big)e(\alpha tn^2)=:\Gamma_1(N)+\Gamma_2(N),
\end{equation*}
where
\begin{align*}
& \Gamma_1(N)=\sum_{N/2<n\leqslant N}\Lambda(n)M_{H_1}(n)M_{H_2}(n)e(\alpha tn^2),\\
& \Gamma_2(N)=\sum_{N/2<n\leqslant N}\Lambda(n)\big(M_{H_1}(n)E_{H_2}(n)+E_{H_1}(n)M_{H_2}(n)
+E_{H_1}(n)E_{H_2}(n)\big)e(\alpha tn^2).
\end{align*}
For each fixed $n\in(N/2,N]$, define
\begin{equation*}
\phi_{n,\gamma}(t)=t^{-1}\left(e\left(t(n+1)^{\gamma}-tn^{\gamma}\right)-1\right),
\qquad S_{n,\gamma}(t)=\sum_{1\leqslant h\leqslant t}e(hn^{\gamma}).
\end{equation*}
It is easy to check that
\begin{align*}
\phi_{n,\gamma}(t)\ll N^{\gamma-1},\qquad\frac{\partial\phi_{n,\gamma}(t)}{\partial t}\ll t^{-1}N^{\gamma-1},
\qquad S_{n,\gamma}(t)\ll \min\left(t,\frac{1}{\|n^{\gamma}\|}\right).
\end{align*}
By partial summation, we obtain
\begin{align}\label{Esti-MH(n)}
           |M_{H_i}(n)|
\ll & \,\, \left|\sum_{1\leqslant |h_i|\leqslant H_i}e\left(h_in^{\gamma_i}\right)\phi_{n,\gamma_i}(h_i)\right|
           \ll\left|\int_{1}^{H_i}\phi_{n,\gamma_i}(t)\mathrm{d}S_{n,\gamma_i}(t)\right|
                \nonumber \\
\ll & \,\, \big|\phi_{n,\gamma_i}(H_i)\big|\big|S_{n,\gamma_i}(H_i)\big|+\int_{1}^{H_i}\big|S_{n,\gamma_i}(t)\big|
           \bigg|\frac{\partial\phi_{n,\gamma_i}(t)}{\partial t}\bigg|\mathrm{d}t
                \nonumber \\
\ll &\,\,  H_iN^{\gamma_i-1}(\log N)\cdot\min\left(1,\frac{1}{H_i\|n^{\gamma_i}\|}\right).
\end{align}
By (\ref{EH(n)}), (\ref{Esti-MH(n)})  and  Lemma \ref{k-min-esti} with $k=2$, we obtain
\begin{align*}
           \Gamma_2(N)
\ll & \,\, (\log N)\sum_{N/2<n\leqslant N}
           \Big(\big|M_{H_1}(n)E_{H_2}(n)\big|+\big|E_{H_1}(n)M_{H_2}(n)\big|+\big|E_{H_1}(n)E_{H_2}(n)\big|\Big)
                   \nonumber \\
\ll & \,\, N^{\frac{14(\gamma_1+\gamma_2)-27}{43}-\frac{\varepsilon}{3}}(\log N)^2
           \times\sup_{(u_1,u_2)\in[0,1]^2}\sum_{N/2<n\leqslant N}
           \prod_{i=1}^2\min\left(1,\frac{1}{H_i\|(n+u_i)^{\gamma_i}\|}\right)
                   \nonumber \\
\ll & \,\, N^{\frac{14(\gamma_1+\gamma_2)-27}{43}-\frac{\varepsilon}{3}}(\log N)^2\left(\frac{N(\log N)^2}{H_1H_2}
           +N^{2/3}(\log N)^2\right)\ll N^{\frac{29(\gamma_1+\gamma_2)-16}{43}+\frac{\varepsilon}{2}}.
\end{align*}
Now, we shall give the upper bound estimate of $\Gamma_1(N)$. Define
\begin{align*}
 \Psi_{h,\gamma}(n)=e\left(h(n+1)^{\gamma}-hn^{\gamma}\right)-1,\qquad
 \Psi(n)=\Psi_{h_1,\gamma_1}(n)\Psi_{h_2,\gamma_2}(n).
\end{align*}
It is easy to check that
\begin{align*}
\Psi(n)\ll |h_1h_2|N^{\gamma_1+\gamma_2-2},\qquad\qquad \frac{\partial\Psi(n)}{\partial n}\ll |h_1h_2|N^{\gamma_1+\gamma_2-3}.
\end{align*}
Accordingly, by partial summation, (\ref{MH(n)}) and Lemma \ref{Applying-H-B-identity}, we derive that
\begin{align*}
          \Gamma_1(N)
 = & \,\, \sum_{N/2<n\leqslant N}\Lambda(n)e(\alpha tn^2)\sum_{1\leqslant|h_1|\leqslant H_1}
          \frac{e(h_1n^{\gamma_1})\Psi_{h_1,\gamma_1}(n)}{2\pi ih_1}\sum_{1\leqslant|h_2|\leqslant H_2}
          \frac{e(h_2n^{\gamma_2})\Psi_{h_2,\gamma_2}(n)}{2\pi ih_2}
                  \nonumber \\
\ll & \,\, \sum_{1\leqslant|h_1|\leqslant H_1}\sum_{1\leqslant|h_2|\leqslant H_2}\frac{1}{|h_1h_2|}
           \Bigg|\sum_{N/2<n\leqslant N}\Lambda(n)e\big(\alpha tn^2+h_1n^{\gamma_1}+h_2n^{\gamma_2}\big)\Psi(n)\Bigg|
                  \nonumber \\
\ll & \,\, \sum_{1\leqslant|h_1|\leqslant H_1}\sum_{1\leqslant|h_2|\leqslant H_2}
           \frac{1}{|h_1h_2|}\Bigg|\int_{\frac{N}{2}}^{N}\Psi(t)\mathrm{d}
           \Bigg(\sum_{N<n\leqslant t}\Lambda(n)e\big(\alpha tn^2+h_1n^{\gamma_1}+h_2n^{\gamma_2}\big)\Bigg)\Bigg|
                  \nonumber \\
\ll & \,\, \sum_{1\leqslant|h_1|\leqslant H_1}\sum_{1\leqslant|h_2|\leqslant H_2}
           \frac{1}{|h_1h_2|}\big|\Psi(N)\big|
           \Bigg|\sum_{N/2<n\leqslant N}\Lambda(n)e\big(\alpha tn^2+h_1n^{\gamma_1}+h_2n^{\gamma_2}\big)\Bigg|
                  \nonumber \\
   & \,\,+\sum_{1\leqslant|h_1|\leqslant H_1}\sum_{1\leqslant|h_2|\leqslant H_2}\frac{1}{|h_1h_2|}
          \int_{\frac{N}{2}}^{N}
     \Bigg|\sum_{N/2<n\leqslant t}\Lambda(n)e\big(\alpha tn^2+h_1n^{\gamma_1}+h_2n^{\gamma_2}\big)\Bigg|
     \bigg|\frac{\partial{\Psi(t)}}{\partial t}\bigg|\mathrm{d}t
                  \nonumber \\
\ll & \,\, N^{\gamma_1+\gamma_2-2}\max_{\substack{1\leqslant|h_1|\leqslant H_1 \\ 1\leqslant|h_2|\leqslant H_2}}
           \max_{N/2<t\leqslant N}\sum_{1\leqslant|h|\leqslant H_1}
           \sum_{1\leqslant|h|\leqslant H_2}\Bigg|\sum_{N/2<n\leqslant t}\Lambda(n)
           e\big(\alpha tn^2+h_1n^{\gamma_1}+h_2n^{\gamma_2}\big)\Bigg|
                  \nonumber \\
\ll & \,\, N^{\gamma_1+\gamma_2-2}\cdot N^{\frac{38+(\gamma_1+\gamma_2)}{43}+\frac{7\varepsilon}{6}}\ll
           N^{\frac{29(\gamma_1+\gamma_2)-16}{43}}.
\end{align*}
This completes the proof of Theorem \ref{Theorem}.

\section*{Acknowledgement}

The authors would like to appreciate the referee for his/her patience in refereeing this paper.
This work is supported by Beijing Natural Science Foundation (Grant No. 1242003), and
the National Natural Science Foundation of China (Grant Nos. 12471009, 12301006, 11901566, 12001047).

\end{document}